\documentclass{smfart}

\usepackage{amsmath,amssymb,enumerate}
\usepackage{epsfig}

\newtheorem{theorem}{Th\'eor\`eme}

\newtheorem{proposition}{Proposition}
\newtheorem{lemma}[proposition]{Lemme}
\newtheorem{corollary}[proposition]{Corollaire}

\newtheorem{definition}[proposition]{D\'efinition}
\newtheorem{remark}[proposition]{Remarque}

\numberwithin{proposition}{section}
\numberwithin{equation}{section}

\newcommand\step[2]{\medbreak\noindent{\bf Etape #1. }{\it #2}\medbreak}

\newcommand{\occult}[1]{}

\newcommand{\new}[1]{{\bf #1}}
\newcommand{\NEW}[1]{{\bf #1}}

\newcommand\eps{\epsilon}
\newcommand\erg{{\operatorname{erg}}}

\newcommand\mult{{\operatorname{mult}}}

\newcommand\Prob{{\operatorname{Prob}}}
\newcommand\Q{{{\mathcal Q}}}

\renewcommand\top{{\operatorname{top}}}

\newcommand\ZZ{{\mathbb Z}}

\begin{document}

\title[Mesures d'entropie maximale de diff\'eomorphismes Anosov]{Une Nouvelle Analyse \\ des Mesures Maximisant l'Entropie \\des Diff\'eomorphismes d'Anosov de Surfaces}

\begin{abstract}
Cette note illustre la strat\'egie de \cite{BuzziPWAH} en obtenant une nouvelle preuve de la multiplicit\'e finie de
la mesure maximisant l'entropie des diff\'eomorphismes d'Anosov, ici dans le cas bi-dimensionel. Cette approche
\'evite en particulier la construction explicite de partitions de Markov.

\medbreak

\noindent {\it Abstract.---} This note illustrates the strategy of \cite{BuzziPWAH} by giving a new proof of the
finite multiplicity of the maximum entropy measure of Anosov diffeomorphisms (here on surfaces). This
approach avoids the explicit construction of Markov partitions.
\end{abstract}

\author{J\'er\^ome Buzzi}

\address{C.N.R.S. \& Universit\'e Paris-Sud\\France}

\email{jerome.buzzi@math.u-psud.fr}

\urladdr{www.jeromebuzzi.com}


\maketitle

\tableofcontents

\section{Introduction}

Nous donnons dans cette note une nouvelle preuve du r\'esultat classique suivant (nous renvoyons \`a \cite{BowenBook} pour
la preuve usuelle passant par la construction explicite de partitions de Markov finies, \`a \cite{BowenSpecif} pour une
approche bas\'ee sur la propri\'et\'e de sp\'ecification et \`a \cite{GL} pour une approche r\'ecente bas\'ee sur l'op\'erateur de transfert):

\begin{theorem}[Sinai]
Soit $f:M\to M$ un diff\'eomorphisme $C^1$ sur une surface compacte $M$ satisfaisant la condition d'Anosov: il existe une d\'ecomposition continue du fibr\'e tangent: $TM=E^s\oplus E^u$ telle que (i) $f'(x).E^{u/s}_x=E^{u/s}_{fx}$; (ii) $\|f'(x)|E^s_x\|\leq\kappa$ et $\|f'(x)^{-1}|E^u_x\|\leq\kappa$ pour une constante $\kappa<1$ et tout $x\in M$.

Alors il existe un nombre fini de \NEW{mesures maximales}, i.e., de mesures de probabilit\'e $\mu$ sur $M$, invariantes et ergodiques par $f$, dont l'entropie $h(f,\mu)$ co\"\i ncide avec l'entropie topologique $h_\top(f)$. Si, de plus,
$f$ est topologiquement transitive, i.e., s'il existe une orbite dense, alors  la mesure d'entropie
maximale est unique.
\end{theorem}

Le but ici est de d\'ecrire une approche plus flexible, \'evitant notamment
la construction explicite de partitions de Markov et pouvant s'\'etendre \`a
des cas non-uniformes. Il s'agit d'une variante simplifi\'ee de l'approche de \cite{BuzziPWAH} qu'on esp\`ere pouvoir adapter aux diff\'eomorphismes entropie-hyperboliques \cite{BuzziICMP2006}. Cette adaptation,
au moins sous une hypoth\`ese de domination et une condition probablement technique d'int\'egrabilit\'e
lisse, sera pr\'esent\'ee dans un article en pr\'eparation.

\subsection{Contexte}

Tout le long de cette note, $f:M\to M$ sera comme ci-dessus. Nous utiliserons les faits standard suivants:
 \begin{enumerate}[(A)]
  \item il existe des feuilletages r\'eguliers $\mathcal F^u,\mathcal F^s$ de $M$ qui sont $f$-invariants et tangents en chaque point \`a $E^u_x, E^s_x$. La r\'egularit\'e des feuilletages est du type suivant. Au voisinage de chaque point $x$ de $M$, il existe un hom\'eomorphisme $\phi^s:]-1,1[^2\to U_x$ tel que: (i) $U_x$ est un voisinage de $x$; (ii) $\phi^s(]-1,1|\times\{y\})$ est une composante connexe de $\mathcal F^s(\phi^s(0,x)\cap U_x$; (iii) $\phi^s|]-1,1|\times\{y\}$ est un $C^1$-diffeomorphisme sur son image.
  \item il existe un nombre $r_0>0$ tel que pour tout $x\in M$:
   $$
      \forall y\in M\quad [\forall n\in\ZZ\; d(f^ny,f^nx)<r_0] \iff y=x.
   $$
  Un tel $r_0>0$ est appel\'e \new{constante d'expansivit\'e de $f$}.
  \item $h_\top(f)>0$.
  \item Il existe $\eps_0>0$, $C_0<\infty$ et $\kappa_0<1$ tels que pour tous $x,y\in M$: pour $d(x,y)<\eps_0$, 
  si $\mathcal F^s(x)=\mathcal F^s(y)$  alors $d(f^kx,f^ky)\leq C_0\kappa_0^k$; si $\mathcal F^u(x)=\mathcal F^u(y)$
  alors $d(f^{-k}x,f^{-k}y)\leq C_0\kappa_0^k$.
 \end{enumerate}

Ces faits sont bien connus. On pourra se reporter \`a \cite{KatokHasselblatt}.

 \subsection{Quelques d\'efinitions}
 
Les d\'efinitions et convention suivantes seront commodes:
 \begin{itemize}
  \item une mesure est sauf mention contraire une mesure de probabilit\'e;
  \item une \new{mesure de grande entropie} est une mesure invariante et ergodique telle que $h(f,\mu)>h_0$ o\`u $h_0<h_\top(f)$ est un param\`etre implicite. Nous d\'enotons par $\Prob_\erg^{h_0}(f)$ l'ensemble de toutes les mesures invariantes et ergodiques d'entropie $>h_0$;
  \item un sous-ensemble $E\subset M$ est dit \new{$h$-n\'egligeable} s'il est de mesure z\'ero pour toutes les mesures de grande entropie: $\exists h_0<h_\top(f)\forall \mu\in\Prob_\erg^{h_0}(f)$, $\mu(E)=0$; 
  \item une propri\'et\'e \new{a lieu $h$-presque partout} si elle a lieu pour tout point de $M$ en dehors d'un ensemble $h$-n\'egligeable;
  \item une estimation a lieu de fa\c{c}on \new{semi-uniforme} si elle a lieu sur
  un ensemble mesurable de mesure uniform\'ement minor\'ee par rapport \`a toute
  mesure de grande entropie.
%  \item un sous-ensemble $E\subset M$ est dit \new{de $h$-mesure $m_0$} si pour tout $m<m_0$, la mesure de $E$ est au moins $m$ pour toute mesure de grande entropie.
 \end{itemize}
Etant donn\'e un ensemble $\Q$ de parties disjointes de $M$, on d\'efinit pour $x\in M$ o\`u cela a un sens:
  \begin{itemize}
  \item $\Q(x)$, l'\'el\'ement de $\Q$ contenant $x$;
  \item le \new{$\Q$-itin\'eraire} de $x$ est $A\in\Q^\ZZ$ tel que $A_n=\Q(f^nx)$ pour tout $n\in\ZZ$.
  \item la \new{vari\'et\'e symbolique $\Q$-stable} et la \new{vari\'et\'e symbolique $\Q$-instable} sont, respectivement:
    $$
       W^s_\Q(x):=\bigcap_{n\geq1} \overline{f^{-n}\Q(f^nx)}
       \text{ et }
       W^u_\Q(x):=\bigcap_{n\geq1} \overline{f^{n}\Q(f^{-n}x)}
        \qquad (x\in M)
    $$
  et
    $$
       W^s(A):=\bigcap_{n\geq1} \overline{f^{-n}A_n}
       \text{ et }
       W^u(A):=\bigcap_{n\geq1} \overline{f^{n}A_{-n}}
         \qquad (A\in\Q^\ZZ).
    $$
  \item $W^s_\Q(x)$ \new{traverse} $\Q$ (ou simplement, traverse) si
 $W^s_\Q(x)$ contient un arc joignant les deux segments oppos\'es du bord
 instable de $\Q(x)$. Cette notion s'\'etend de fa\c{c}on naturelle \`a
 $W^s(A)$, $W^u_\Q(x)$, $W^u(A)$.
 \end{itemize}
 
\subsection{Strat\'egie de la preuve}
La preuve se d\'ecompose en les \'etapes suivantes:
 \begin{itemize}
  \item Construction d'une partition finie $\Q$ de diam\`etre maximal arbitrairement petit en carr\'es (voir d\'efinition ci-dessous);
  \item Existence, pour tout $m<1$, de $r_1>0$ tel que l'ensemble des $x\in M$ v\'erifiant
  \begin{equation}\label{eq-long-Wsu}
    d(x,\partial W^s_\Q(x))>r_1 \text{ et } d(x,\partial W^u_\Q(x))>r_1
  \end{equation}
  est de mesure au moins $m$ pour toute mesure de grande entropie.
  \item Remplacement de $f$ par une extension p\'eriodique et modifications de $\Q$ garantissant que  $W^s_\Q(x)$ et $W^u_\Q(x)$ traversent pour tout $x$
appartenant \`a un ensemble de mesure strictement positive pour toute mesure de grande entropie. 
  \item Structure markovienne de l'ensemble des $\Q$-itin\'eraires de $h$-presque tout point.
  \item Th\'eorie de Gurevi\v{c} des sous-d\'ecalages markoviens (\`a ensemble d'\'etats d\'enombrable) ramenant le 
d\'enombrement des mesures maximales \`a celui des parties irr\'eductibles d'entropie maximale et contr\^ole de celles-ci.
 \end{itemize}

\section{Une partition $\Q$ en carr\'es et ses vari\'et\'es symboliques}

On construit dans cette section une partition $\Q$ et on obtient une borne inf\'erieure
semi-uniforme pour les vari\'et\'es $\Q$-stables et $\Q$-instables. Les feuilletages invariants permettent
de ne consid\'erer que des g\'eom\'etries triviales:

\begin{definition}
Un \new{carr\'e} (ou carr\'e $us$) est un disque ouvert topologique dont le bord est constitu\'e de quatre segments
de courbes compacts alternativement inclus dans une feuille stable et une feuille instable (voir Figure \ref{fig-carre-us}).
On parlera des deux bords stables et des deux bords instables.
\end{definition}

\begin{figure}
\centering
\includegraphics[width=4cm]{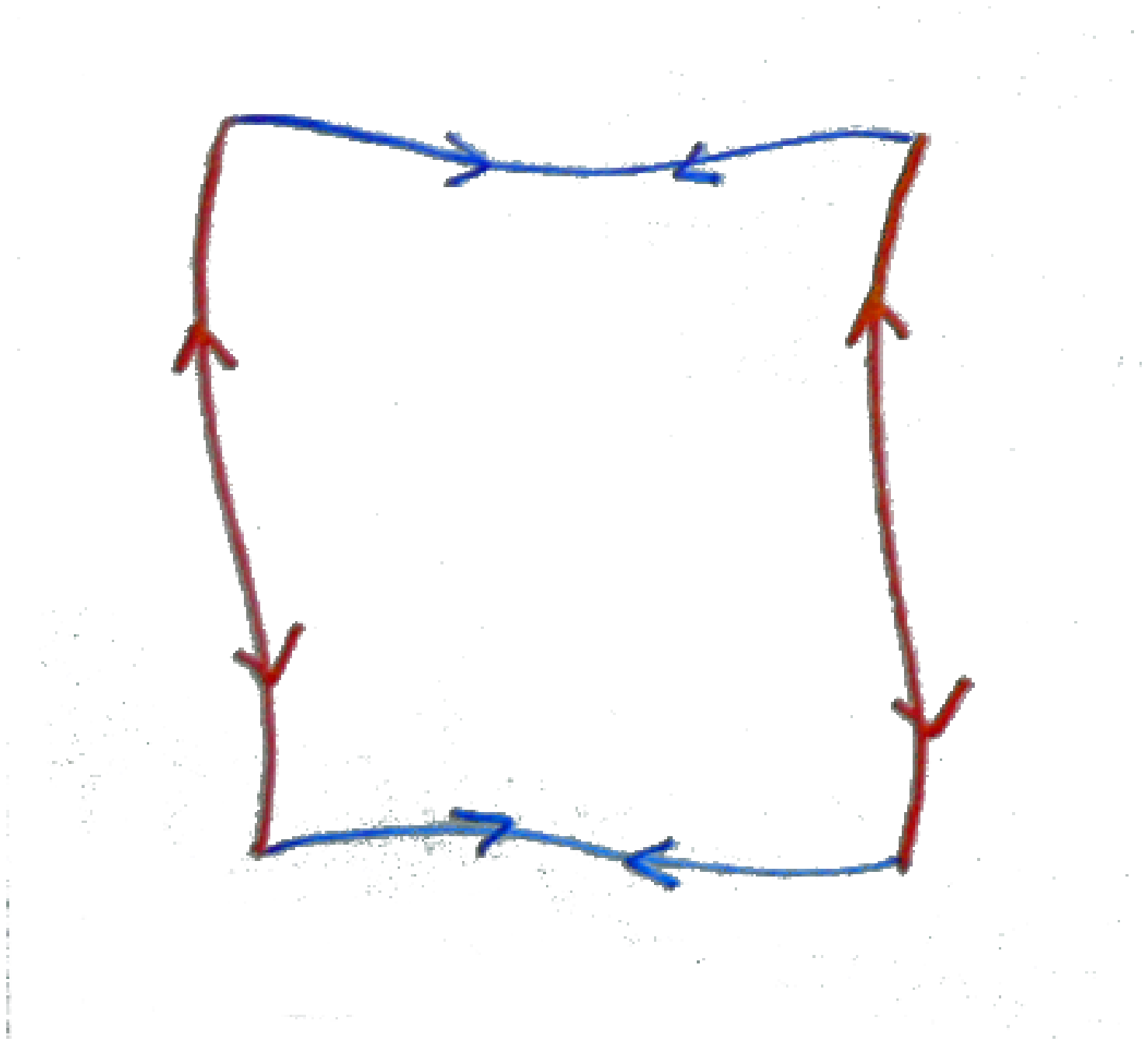}
% where an .eps filename suffix will be assumed under latex,
% and a .pdf suffix will be assumed for pdflatex
\caption{Un carr\'e born\'e par quatre courbes alternativement incluses dans $\mathcal F^s$ et $\mathcal F^u$.}\label{fig-carre-us}
\end{figure}

Il est imm\'ediat que l'image ou l'intersection de tels carr\'es sont encore des carr\'es du m\^eme type.
Le lemme suivant est clair:

\begin{lemma}\label{lem-find-Q-r} 
Soit $r>0$. Il existe une partition finie en carr\'es dont les diam\`etres sont major\'es par $r$.
\end{lemma}

L'estimation semi-uniforme annonc\'ee est la suivante:

\begin{proposition}\label{prop-Wu-size}
Pour tout $m_1>0$, il existe $h_1<h_\top(f)$ et $r_1>0$ tel que pour tout $\mu\in\Prob_\erg^{h_0}(f)$,
 $$
   \mu\left(\{x\in M:d(x,\partial W^u_\Q(x))<r_1\}\right) < m_1.
 $$
\end{proposition}

\begin{remark}
Si une mesure ergodique est d'entropie non-nulle, la preuve ci-dessous impliquement seulement qu'il existe un ensemble de mesure positive o\`u $d(x,\partial W^s_\Q(x))\geq\text{const}$ et un autre, \'egalement de mesure positive, o\`u
$d(x,\partial W^u_\Q(x))\geq\text{const}$, mais n'entra\^\i ne pas que ces deux ensembles
s'intersectent. Ceci est un obstacle 
\`a l'analyse des mesures d'\'equilibre pour des potentiels non presque constants.
\end{remark}

Avant de prouver cet \'enonc\'e, donnons quelques d\'efinitions et faits auxiliaires. Les \new{partitions it\'er\'es} sont, pour $n\geq1$, 
 $$
  \Q^n:=\{[A_0A_1\dots A_{n-1}]:=A_0\cap f^{-1}A_1\cap\dots\cap f^{-n+1}A_{n-1}\ne\emptyset:A_0,\dots,A_{n-1}\in\Q\}. 
 $$
La \new{multiplicit\'e locale} d'une collection $C$ de parties de $M$ est:
 $$
    \mult(C) = \max_{x\in M} \#\{A\in C:\bar A\ni x\}.
 $$
Pour tout $n\geq1$, $\Q^n$ est une partition en carr\'es du fait de l'invariance des feuilletages. Il s'en suit que:
 \begin{equation}\label{eq-mult-4}
   \forall n\geq0\; \mult(\Q^n) \leq 4.
 \end{equation}

La formule suivante est classique (voir, par exemple, \cite{Petersen}):

\begin{lemma}\label{lem-h-by-W}
Soit $\mu$ une mesure $f$-invariante. Consid\'erons la partition mesurable $\{W^u_\Q(x):x\in M\}=\bigvee_{n\geq1} f^n\Q$ et la d\'esint\'egration
associ\'ee $\{\mu|W^u_\Q(x)\}_{x\in M}$ of $\mu$. Alors, pour tout $N\geq1$,
 \begin{align*}
   h(f,\mu,\Q) &= \frac1N H_\mu(\Q^N|\bigvee_{n\geq1} f^n\Q) = \frac1N \int_M H_{\mu|W^u_\Q(x)}(\Q^N) \, d\mu \\
    &\leq \int_M \frac1N\log\#\{A\in\Q^N:(\mu|W^u(\Q(x)))([A])>0\} \, \mu(dx)
 \end{align*}
\end{lemma}

Remarquons que, vu l'expansivit\'e de $f$ et le petit diam\`etre de $\Q$, $h(f,\mu)=h(f,\mu,\Q)$.

\begin{proof}[Preuve de la Proposition]
Soit $0<\alpha<m_1\cdot h_\top(f)$. Fixons $T>\log 8/\alpha$ tel que:
 $$
   \forall n\geq T\quad \#\Q^n \leq e^{(h_\top(f)+\alpha)n}.
 $$
D'apr\`es (\ref{eq-mult-4}) et la compacit\'e, pour $r_1>0$ assez petit, une boule ferm\'e de rayon $r_1$ rencontre au plus $4$ 
\'el\'ements de $\Q^T$. Alors, pour tout $A\in\Sigma(f,\Q)$,
 $$
   \{\Q^T(x):W^u_\Q(x)=W^u(A) \text{ et }d(x,\partial W^u_\Q(x))<r_1) \leq 2\times 4
 $$
En effet, une fois connu le point extr\'emit\'e $z$ proche de $x$ (un choix binaire \'etant donn\'e le pass\'e), $\Q^T(x)$ se trouve parmi au plus $4$ \'el\'ements de $\Q^T$ rencontrant $\overline{B(z,r_1)}$.

Soit $\mu$ une mesure ergodique et invariante. Posons $M_1:=\{x\in M:d(x,\partial W^u_\Q(x))\leq r_1\}$.
Le lemme \ref{lem-h-by-W} donne:
 \begin{multline*}
   h(f,\mu) \leq \int_{M_1} \frac1T \log 8 \,d\mu \\
     + \int_{M\setminus M_1}(h_\top(f)+\alpha)
      \leq h_\top(f)-\mu(M_1)h_\top(f)+\alpha:= h_1
 \end{multline*}
Si  $\mu(M_1)\geq  m_1 >\frac{\alpha}{h_\top(f)}$, alors on obtient une borne non-triviale:
 $$
   h(f,\mu)\leq h_1:=h_\top(f)-m_1h_\top(f)+\alpha<h_\top(f).
 $$
\end{proof}

\section{Dynamique symbolique}

La partition $\Q$ pr\'ec\'edemment construite permet de d\'efinir une dynamique
symbolique dont on montre ci-dessous qu'elle repr\'esente fid\`element les mesures
maximales de $f$.

\begin{definition}
Si $\mathcal P$ est une collection d'ouverts disjoints et d'union dense $M'$ dans $M$, alors la \new{dynamique symbolique} est d\'efinie comme
le sous-d\'ecalage
 $$
    \Sigma(f,\mathcal P) := \overline{\{\dots\mathcal P(f^{-1}x)\mathcal P(x)\mathcal P(fx)\dots:x\in M'\}}\quad \sigma((A_n)_{n\in\ZZ}) =
   (A_{n+1})_{n\in\ZZ}
 $$
o\`u la fermeture est prise dans $\mathcal P^\ZZ$, $\mathcal P$ \'etant muni de la topologie discr\`ete. La projection 
$p:\Sigma(f,\mathcal P)\to M$ est d\'efinie par $\{p(x)\}=\bigcap_{n\in\ZZ} \overline{A_n}$, o\`u cela a un sens.
\end{definition}

On v\'erifie ais\'ement que $p\circ\sigma = f\circ p$ o\`u cela a un sens et que
 $$
   \Sigma(f,\mathcal P) = \{ A\in\mathcal P^\ZZ: \forall n<m\; [A_n\dots A_m]\ne
     \emptyset\}.
 $$

\begin{proposition}
Soit $\Q$ une partition\footnote{Abus de language: $\Q$ est en fait une collection d'ouverts disjoints d'union dense dans $M$.} de $M$ en l'int\'erieurs de carr\'es $us$ de diam\`etres strictement plus petits que la constante d'expansivit\'e.
Soit $\Sigma(f,\Q)$ et $p:\Sigma(f,\mathcal P)\to M$ comme ci-dessus.
 \begin{itemize}
   \item $p$ est bien d\'efinie sur tout $A\in\Sigma(f,\mathcal P)$ et continue.
   \item $p$ est au plus $4$ sur $1$. En particulier, toute mesure de probabilit\'e invariante se rel\`eve et $p$ 
     pr\'eserve l'entropie mesur\'ee et topologique.
   \item Soit $\mu,\mu'$ des mesures de probabilit\'e de $\Sigma(f,\mathcal P)$. Supposons $p_*\mu$ ap\'eriodique. Alors $p_*\mu=p_*\mu'$ implique $\mu=\mu'$ et  $p_*\mu$ est ergodique si et seulement si $\mu$ est ergodique.
 \end{itemize}
\end{proposition}

En particulier, on a imm\'ediatement:

\begin{corollary}
$p$ induit une bijection entre les mesures maximales de $f$ et de $\Sigma(f,\Q)$.
\end{corollary}

Le lemme suivant est utile. 

\begin{lemma}
Soit $\mu$ une mesure invariante et ergodique. Si $\mu(\mathcal F^s(x))>0$ (plus pr\'ecis\'ement: s'il existe une partie mesurable de $\mathcal F^s(x_0)$ de mesure non-nulle) alors $\mu$ est p\'eriodique.
\end{lemma}

\begin{proof}
Par r\'ecurrence, on peut trouver $z\in\mathcal F^s(x)$ et $p\geq1$ tels que, pour tout $r>0$, $\mu(B(z,r)\cap\mathcal 
F^s(z))>0$ et $f^pz$ est proche de $z$.  (D, page 1) montre alors que $f^p$ est une contraction de $B(z,r)\cap\mathcal 
F^s(z)$ dans  lui-m\^eme. Le th\'eor\`eme ergodique montre alors que $\mu$ est une mesure p\'eriodique port\'ee par 
l'orbite de $p$. 
\end{proof}

\begin{proof}[Preuve de la proposition]
D'une part, $f$ \'etant expansif et $\Q$ de petit diam\`etre, $p$ est bien d\'efinie et continue. D'autre part,
 $$
    \# p^{-1}(x) \leq \sup_{n\geq1} \mult(\Q^n) = 4.
 $$
La pr\'eservation de l'entropie est \'evidente.
La fibre de $p$ \'etant en particulier compacte toute mesure invariante de $f$ se rel\`eve en une mesure invariante de $\Sigma(f,\mathcal P)$.

Finalement, supposons $p(A)=p(B)$. Si $A\ne B$, il existe $n\in\ZZ$ tel que $A_n\ne B_n$ et $f^nx\in \overline{A_n}\cap\overline{B_n}\subset\partial\mathcal P$. Donc $p(\mu)=p(\mu')$ implique, si $\mu\ne \mu'$,
que $\mu\left(\partial\mathcal P\right)>0$. Mais $\partial\mathcal P$ est une union finie de segments
de feuilles stables et instables. Donc $\mu(\mathcal F^\sigma(x_0))>0$ pour un certain $x_0\in M$ et $\sigma=s$ ou $u$.
Le lemme ci-dessus montre qu'alors $\mu$ est p\'eriodique, une contradiction.
Si $p_*\mu$ est ergodique alors les composantes ergodiques de $\mu$ sont autant de rel\`evements de $p_*\mu$. D'apr\`es
l'unicit\'e de ce rel\`evement, $\mu$ est \'egalement ergodique. L'implication r\'eciproque est un fait totalement 
g\'en\'eral.
\end{proof}
 
\section{Extension p\'eriodique}

On a vu que les vari\'et\'es symboliques stables et instables sont de taille
semi-uniforme. On va montrer qu'on peut faire en sorte que ces vari\'et\'es symboliques traversent bien $\Q$. La difficult\'e est que la borne inf\'erieure donn\'ee par la proposition \ref{prop-Wu-size} est typiquement plus petite que le diam\`etre de la partition \`a moins que $f$ ne soit fortement dilatante.  Nous utilisons la construction suivante.

Soit un entier $T\geq1$. On pose $M_T:=M\times\ZZ/T\ZZ$ et $f_T(x,k):=(fx,k+1)$. 

L'extension $\pi:( M_T, f_T)\to(M,f)$ est continue et compacte. Donc toute collection de mesures $f$-invariantes et distinctes se rel\`eve en une collection de mesures $f_T$-invariante et distinctes. L'extension est finie donc pr\'eserve l'entropie mesur\'ee. Ainsi il suffit de prouver le th\`eor\`eme pour $f_T$, pour un $T\geq1$ commode.

\begin{proposition}\label{prop-ext-Wsu-cross}
Pour $T$ assez grand, il existe une partition $Q_T$ de $M_T$, partition en carr\'es $us$, telle que, 
 \begin{equation}\label{eq-BT}
    B_T:=\{ x\in M_T: W^s_{\Q_T}(x) \text{ et }  W^u_{\Q_T}(x) \text{ traversent } \Q(x)\}
 \end{equation}
est de mesure $>0$ par rapport \`a toute mesure $f_T$-invariante de grande entropie.
\end{proposition}

\begin{proof}
Le lemme \ref{lem-find-Q-r} fournit une partition $\Q_0$ en carr\'es de diam\`etre $<r_0$.

Soit $0<\eps_1<1$. La proposition \ref{prop-Wu-size} appliqu\'ee deux fois donne $r_1>0$ et $h_1>0$ tels que, sur un ensemble $B_0$ de mesure au moins $1-2\eps_1$ pour toute mesure ergodique d'entropie $>h_1$, $d(x,\partial W^s_{\Q_0}(x))>r_1$ et $d(x,\partial W^u_{\Q_0}(x))>r_1$. 

Soit $\Q_1$ une nouvelle partition finie de $M$ en carr\'es de diam\`etre $<r_1$. Pour $x\in B_0$, $W^u_{\Q_0}(x)$ traverse $\Q_1$.
On passe \`a l'extension pour un entier $T>>1$. On pose
 $$
    \Q_T:=\Q_1\times\{0\}\cup \Q_0\times\{1\}\cup\dots\cup \Q_0\times\{T-1\}
 $$
Le lemme ci-dessous donne $h_2<h_\top(f)$ tel que, en posant $B_T^0:=B_T\cap (M\times\{0\})$ et en remarquant que $x\in M\times\{0\}\setminus B^0_T \subset
(M\setminus B_0)\times\{0\}\cup C_T^0$:
 $$
    \mu_T(B^0_T)>\frac{1-4\eps_1}{T} \implies h(f_T,\mu_T)<\max(h_1,h_2)
 $$
avec $\max(h_1,h_2)<h_\top(f)$, d\`es que $T$ est choisi assez grand.
\end{proof}

\begin{lemma}
Pour tout $\eps_1>0$, il existe $T_1<\infty$ et $h_2<h_\top(f)$ tels que pour tout $T\geq T_1$ et toute mesure
ergodique satisfaisant $h(f_T,\mu_T)>h_2$:
 $$
   \mu_T(C^0_T)<\frac{\eps_1}T
   \text{ o\`u }
      C^0_T:=\{(x,0)\in M\times\{0\}:\pi W^u_{\Q_T}(x,0)\subsetneq W^u_\Q(x) \}.
 $$
\end{lemma}

\begin{proof}
Supposons $\pi W^u_{\Q_T}(x,0)\subsetneq W^u_\Q(x)$. Le raccourcissement de $W^u_{\Q_T}(x,0)$ ne peut venir que d'un temps
$k<0$ tel que $f_T^k(x,0)\in M\times\{0\}$. Plus pr\'ecis\'ement, il doit exister un entier $n\geq1$ tel que
$\partial\Q_1(f^{-nT}x)$ coupe $f^{-nT}W^u_{\Q_0}(x)$ en un point $z$.
On a donc $\Q_0^{nT}(z)=\Q_0^{nT}(x)$ pour un $z\in W^u_{\Q_0}(f^{-nT}x)\cap \partial\Q_1(f^{-nT}x)$. 

%On peut supposer que $W^u_{\Q_0}(x)$ traverse $\Q_0(x)$  pourvu que $\mu_T$ ne soit pas p\'eriodique et que $x$ n'appartienne pas \`a un ensemble n\'egligeable.

On voit que pour tout $(x,0)\in C^0_T$, il existe un entier $n=n(x)\geq1$ tel que, sachant $W^u_{\Q_0}(f^{-nT}x)$  et $\Q_1(f^{-nT}x)$, $\Q_0^{nT}(f^{-nT}x)$ est d\'etermin\'e par le choix de $z$ parmi les deux points d'intersection de $W^u_{\Q_0}(f^{-nT}x)$ avec $\partial\Q_1(f^{-nT}x)$. On d\'eduit une borne sur l'entropie.

Soit $\alpha<\eps_1 h_\top(f)/16$. En fixant $T_1$ assez grand on peut garantir:
 $$
    \forall n\geq T_1\;\#\Q_0^n \leq e^{(h_\top(f)+\alpha)n}.
 $$
On peut supposer $T_1\geq \log 2\#\Q_1\log\#\Q_0/\alpha\geq \log 2/\alpha$.
Fixons $T\geq T_1$ et une mesure $\mu_T$. On suppose par contradiction que:
 $$
   m:=T\cdot\mu_T\left(C_T^0\right)\geq \eps_1.
 $$ 
Fixons $N=N(\mu)<\infty$ tel que
 $$
    \mu_T(\{x\in C^0_T:n(x)>N\})<\frac{\beta}{T}.
 $$
Le th\'eor\`eme de Birkhoff fournit un ensemble $E_2\subset M_T$ et un entier $N_2$ tels que 
$\mu_T(E_2)>1-\eps_1/T$ et pour tout $x\in E_2$, tout $n\geq N_2$:
 $$
    \left|\frac1n\#\{0\leq k<n:f_T^{kT}(x)\in C^0_T \} -m\right|< \frac{\eps_1}{8}
    \text{ et }
    \frac1n\#\{0\leq k<n:n(f_T^{kT}(x))>N \} < \beta
 $$
On d\'ecoupe $[0,n-1]$ en des sous-intervalles $[aT,bT-1]$, $a\leq b$ entiers, de trois types:
 \begin{itemize}
  \item $f_T^{bT}x\in B^0_T$, $n(f_T^{bT}x)\leq N$ et $a=b-n(x)$ ($[aT;bT-1]$ est de type 1);
  \item $f_T^{bT}x\notin B^0_T$ et $a=b-1$ ($[aT;bT-1]$ est de type 2);
  \item $f_T^{bT}x\in B^0_T$, $n(f_T^{bT}x)>N$ et $a=b-1$ ($[aT;bT-1]$ est de type 3).
 \end{itemize}
On d\'ecrit $\Q_0^n(x)$ de la fa\c{c}on suivante. On sp\'ecifie la position des intervalles
de type $1$ parmi les $n/T$ intervalles de longueur $T$ (ce qui donne $2^{n/T}$
choix au plus). Plus pr\'ecis\'ement il y en a entre $\eps_1 n/2T$ et $2\eps_1 n/T$
et leur longueur totale est $L_1\geq \eps_1n/2$. 

Chaque intervalle de type $1$ correspond
\`a une suite $\Q_0^{(b-a)T}$ \`a choisir parmi $2\cdot\#\Q_1$ suites seulement,
si on a d\'ej\`a d\'etermin\'e les  $\Q_0(f^kx)$ pour $k\geq aT$, d'apr\`es les remarques pr\'ec\'edentes. On a donc,
conditionnellement, au plus $(2\cdot\#\Q_1)^{2\eps_1 n/T}$ choix de symboles pour l'ensemble de ces intervalles.

Pour les intervalles de type 2 ou 3, $\Q_0^{(b-a)T}(f^ax)\in
\Q_0^{\ell T}$, $\ell\geq1$. On a donc $e^{(n-L_1)(h_\top(f)+\alpha)}$ choix possibles
au total pour ces intervalles.
 
Au final le nombre de $\Q_0^n(x)$ sachant
$W^u_{\Q_0}(x)$ est born\'e par:
 $$
    2^{n/T} \times  (2\cdot\#\Q_1)^{2\eps_1 n/T} \times e^{n(1-\eps_1/2))(h_\top(f)+\alpha)}
     \leq \exp n\left(h_\top(f)-\frac{\eps_1}2 h_\top(f)+3\alpha\right)
 $$
Donc, d'apr\`es le lemme \ref{lem-h-by-W}, on a:
 $$
    h(f,\mu) \leq h_\top(f)-\frac{\eps_1}4 h_\top(f)=:h_1.
 $$
\end{proof}

\section{Structure markovienne}

D'apr\`es la Proposition \ref{prop-ext-Wsu-cross}, on peut supposer que
 $$
       B:=\{ x\in M: W^s_{\Q}(x) \text{ et }  W^u_{\Q}(x) \text{ traversent }\Q(x) \}
 $$
est de mesure non-nulle pour toute mesure de grande entropie.

Soit $\tau_B(x):=\inf \{n\geq1:f^nx\in B\}$, le temps de premier retour dans $B$.
D'apr\`es le th\'eor\`eme de Kac, $h$-presque tout point de $\Sigma(f,\Q)$ passe donc une
infinit\'e de fois dans $B$ aussi bien dans le futur que dans le pass\'e. Soit
$\Sigma'(f,\Q)$ l'ensemble des $\Q$-itin\'eraires de ces bons points.

\begin{definition}
Un \new{mot de premier retour} est une suite finie $A_0\dots A_n$ telle qu'il existe
$x\in B$ satisfaisant: $A_k=\Q(f^kx)$ (pour tout $0\leq k\leq n$)
et $n=\tau_B(x)$. On lui associe la suite finie $(A_0,*)A_1\dots A_{n-1}$ 
appel\'ee \new{mot de base}. 

Le sous-d\'ecalage $\Sigma_B$ est l'ensemble des suites bi-infinies form\'ees par la 
concat\'enation de mots de base sous la condition suivante:
\begin{quote}
  Si $(A_0,*)A_1\dots A_{n-1}(B_0,*)$ appara\^{\i}t alors $A_0\dots A_{n-1}B_0$ est un
  mot de premier retour.
\end{quote}
Pour une suite $\hat A\in\Sigma_B$, les entiers $k\in\ZZ$ dont le $k$\`eme symbole $\hat A_k$ est \'etoil\'e sont appel\'es des \new{temps marqu\'es}.
\end{definition}

Remarquons que $\Sigma_B$ ainsi d\'efini a un alphabet fini mais n'est pas forc\'ement
un sous-shift et que son adh\'erence n'est pas n\'ecessairement de type fini. Toutefois,
on verra que $\Sigma_B$ est conjugu\'e \`a un sous-d\'ecalage markovien  (sur un alphabet
d\'enombrable, infini en g\'en\'eral) et ceci permettra son analyse.

\begin{proposition}
Soit $\pi:\Sigma_B\to\Sigma(f,\Q)$ la projection d\'efinie symbole par symbole par $\pi(A,*)=A$ et $\pi(A)=A$ pour
tout $A\in\Q$. Alors $\pi$ est bien d\'efinie et induit une bijection entre les mesures maximales de
$\Sigma_B$ et celles de $\Sigma(f,\Q)$.
\end{proposition}

\begin{figure}
\centering
\includegraphics[width=10cm]{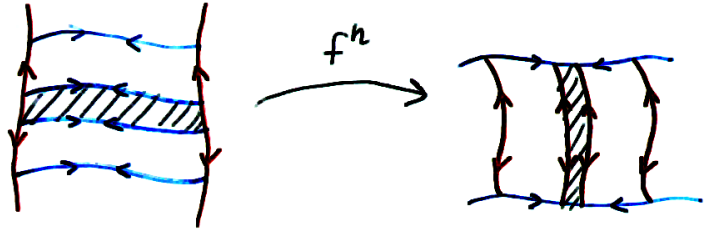}
% where an .eps filename suffix will be assumed under latex,
% and a .pdf suffix will be assumed for pdflatex
\caption{Un rectangle topologiquement hyperbolique et son image (hachur\'es), repr\'esent\'es dans leur carr\'es respectifs. Les courbes approximativement horizontales (bleues), resp. verticales (rouges), sont des morceaux de $\mathcal F^s$, resp. $\mathcal F^u$.}\label{fig-hyp-top}
\end{figure}

\step1{$\pi$ est bien d\'efinie.}

Il faut voir que l'image de tout $\omega\in\Sigma_B$ est bien dans $\Sigma(f,\Q)$, c'est-\`a-dire que
$\bigcap_{k=-n}^m f^{-k}A_k\ne\emptyset$ pour des entiers $n,m\to\infty$. Il suffit de montrer que si
$A_0^i\dots A_{n_i}^i$, $i=1,\dots,I$, sont des mots de premiers retours avec $A_{n_i}^i=A_0^{i+1}$ pour
tout $i=1,\dots,I-1$, alors,
 $$
     \bigcap_{i=1}^I f^{-(n_1+\dots+n_{i-1})} H_i \ne \emptyset
     \quad \text{ o\`u } H_i:=\bigcap_{k=0}^{n_i} f^{-k}A_k^i.
 $$
Ceci provient de l'hyperbolicit\'e "topologique" des $f^{n_i}:H_i\to f^{n_i}H_i$ illustr\'ee par la 
figure \ref{fig-hyp-top}. On v\'erifie par r\'ecurrence que l'intersection ci-dessus est encore
hyperbolique et en particulier est non vide.

\step2{$\pi$ est injective}

Soit $A\in\Sigma'(f,\Q)$. Soit $\hat A\in\Sigma_B$ le relev\'e obtenu en marquant exactement les temps $n\in\ZZ$ tels que
$\sigma^n A\in B_\Sigma$. $\hat A$ est clairement un \'el\'ement bien d\'efini de $\Sigma_B$.

Il suffit de montrer que si $\hat A'$ est un rel\`evement quelconque de $A$, $\hat A'=\hat A$.

Pour tout temps marqu\'e $k$ de $\hat A'$, le raisonnement de l'\'etape 1 implique que $W^s(\sigma^kA)$ et $W^u(\sigma^kA)$ traversent. Donc un tel $k$ est marqu\'e pour $\hat A$. Montrons la r\'eciproque
par l'absurde: on suppose que $0,n_1,n_2,\dots,n_r$, $r\geq2$, sont des temps marqu\'es cons\'ecutifs dans $\hat A$ et que, parmi ceux-ci, seuls $0$ et $n_r$ sont marqu\'es dans $\hat A'$.

$A_0\dots A_{n_r}$ est donc $\Q^{n_r+1}(y)$ pour un $y\in B$ avec $\tau_B(y)=n_r$. Il s'en suit que 
$W^s_\Q(f^{n_1}y)$ ne traverse pas. Autrement dit, l'intersection $\ell^s$ 
de la courbe stable locale avec $\Q(f^{n_1}y)$ n'est pas contenue dans $W^s_\Q(f^{n_1}y)$. 
Donc $f^m\ell^s$ n'est pas inclus
dans un \'el\'ement de $\Q$ pour un certain $m>0$ qu'on peut choisir minimal. 
N\'ecessairement $m<n_r-n_1$: sinon $W^s_\Q(f^{n_r}y)$ qui traverse contiendrait $f^{n_r-n_1}(\ell^s)$, impliquant que $\ell^s\subset W^s_\Q(f^{n_1}y)$, une contradiction.

Le bord instable de $\Q(f^my)=A_m$ raccourcit donc $W^s_\Q(f^{n_1}y)$ par rapport \`a $\ell^s$. Un des bords instables de $[A_{n_1}\dots A_{n_1+m-1}]$ est donc envoy\'e \`a l'ext\'erieur du bord instable de $A_{n_1+m}$. Mais ces bords ne se coupent ni ne se contournent, donc  $W^s_\Q(\sigma^{n_1}A)$ serait aussi raccourci, 
en contradiction avec la d\'efinition de $n_1$.

\section{Conclusion}

D'apr\`es ce qui pr\'ec\`ede, il suffit de prouver la multiplicit\'e finie ou l'unicit\'e des mesures maximales
pour le d\'ecalage $\Sigma_B$. Soit $\hat\Sigma$ le d\'ecalage $\sigma$ sur l'ensemble des chemins bi-infinis sur le graphe $\mathcal D$ dont les sommets sont $(w,k)$ o\`u $w$ est un mot de premier retour et $1\leq k<|w|$, $|w|$ d\'esignant la longueur
de $w$ et les fl\`eches sont $(w,k)\to(w,k+1)$ et $(w,|w|-1)\to(w',1)$ si le dernier symbole de $w$ est
le premier symbole de $w'$. Un th\'ero\`eme de Gurevic dit qu'il suffit de prouver que $\mathcal G$ a un nombre fini, resp. \'egal \`a $1$, de composantes irr\'eductibles.

Observons que, par construction, toute composante irr\'eductible de $\mathcal D$ contient un \'etat \'etoil\'e
et ceux-ci sont en nombre fini, prouvant l'assertion principale. 

Il faut maintenant montrer l'unicit\'e dans le cas o\`u $f$ est topologiquement
transitive. Remarquons que la transitivit\'e de $f$ est \'equivalente au fait que l'orbite 
de tout ouvert non-vide de toute feuille instable (ou stable) est dense. En effet, un tel
ouvert contient une courbe $W^u_\eps(x)$ pour $\eps>0$ assez petit or il existe $z\in B(x,\eps/2)$ dont l'orbite
est dense. Comme $W^s_\eps(z)$ rencontre $W^u_\eps(x)$, ceci prouve la remarque, comme le diam\`etre de $W^s_\eps(x)$
tend vers z\'ero d'apr\`es (D, page 1).

Les pr\'eorbites des points p\'eriodiques forment un ensemble d\'enombrable: $N_1:=\bigcup_{n\geq0} f^{-n}(\operatorname{Per}(f))$. On 
peut donc choisir la partition de sort que (*) ses bords stables ne rencontrent pas $N_2:=\bigcup_{x\in N_1} W^s_\eps(x)$.
 
Pour montrer l'irr\'eductibilit\'e de $\mathcal D$, il suffit de v\'erifier que pour deux carr\'es quelconques $R$ et $R'$ 
de la partition,  tels qu'il existe une boucle $\Gamma$ bas\'ee en $R'$, on peut trouver une concat\'enation admissible
de mots de base, le premier commen\c{c}ant par $R$ et le dernier finissant par $R'$.

$\Gamma$ fournit un point $T$-p\'eriodique $p$ tel que, pour tout $\eps>0$, $f^T (W^u_\eps(p)\cap \Q^n(p))\supset 
W^u_\eps(p)$.  (*) assure que $p$ est \`a l'int\'erieur de $R'$. Soit $W^u_r(x_0)\subset R$ dont les extr\'emit\'es
\'evitent $N_2$ (**). La transitivit\'e 
donne $n\geq0$ tel que $f^nW^u_r(x)$ contient $W^u_\eps(y)\cap R'$ pour un $y\in R'$. On peut supposer $y\in W^s_\eps(p)$.
$f^n(W^u_r(x)\cap\Q^n(x))=I$, un intervalle de $W^u(y)$ contenant $y$ en son int\'erieur d'apr\`es (*) et (**). La dilatation
le long de $\mathcal F^u$ (voir D, page 1) garantit maintenant que $f^{n+kT}(W^u_r(x)\cap \Q^n(x)\cap f^{-n}\Q^{kT})$ traverse $R'$ pour $k$ assez grand. Mais ceci dit exactement
que le $Q,n+kT$-itin\'eraire de $x$ est une concat\'enation de mots de premiers retour, comme souhait\'e. $\mathcal D$ est bien
irr\'eductible et la preuve du th\'eor\`eme est achev\'ee.

\end{document}